\documentclass[12pt, reqno]{amsart}
\usepackage{amsmath, amsthm, amscd, amsfonts, amssymb, graphicx, color}
\usepackage[bookmarksnumbered, colorlinks, plainpages]{hyperref}

\newtheorem{theorem}{Theorem}[section]
\newtheorem{lemma}[theorem]{Lemma}

\theoremstyle{definition}

\theoremstyle{remark}
\newtheorem{remark}[theorem]{Remark}
\numberwithin{equation}{section}

\begin{document}
\setcounter{page}{1}

\title[2-Local derivations on matrix rings]{2-Local derivations on
associative and Jordan matrix rings over commutative rings}

\author[Sh.\ Ayupov]{Shavkat Ayupov$^1$}

\address{$^{1}$
Institute of Mathematics, National University of Uzbekistan,
Tashkent, Uzbekistan.}
\email{\textcolor[rgb]{0.00,0.00,0.84}{sh$_-$ayupov@mail.ru}}

\author[F. Arzikulov]{Farhodjon Arzikulov$^2$}

\address{$^{2}$ Department of Mathematics, Andizhan State University, Andizhan, Uzbekistan.}
\email{\textcolor[rgb]{0.00,0.00,0.84}{arzikulovfn@rambler.ru}}




\thanks{}

\begin{abstract}
In the present paper we prove that every 2-local inner
derivation on the matrix ring over a commutative
ring is an inner derivation and
every derivation on an associative ring has an extension to a derivation on
the matrix ring over this associative ring
We also develop a Jordan analog of the above method and
prove that every 2-local inner
derivation on the Jordan matrix ring over a commutative
ring is a derivation.
\end{abstract}

\maketitle

\section{Introduction}

The present paper is devoted to 2-local derivations on associative
and Jordan matrix rings. Recall that a 2-local derivation is defined as follows:
given a ring $\Re$, a map $\Delta : \Re \to \Re$ (not additive in
general) is called a 2-local derivation if for every $x$, $y\in
\Re$, there exists a derivation $D_{x,y} : \Re\to \Re$ such that
$\Delta(x)=D_{x,y}(x)$ and $\Delta(y)=D_{x,y}(y)$.

In 1997, P. \v{S}emrl \cite{S} introduced the notion of 2-local
derivations and described 2-local derivations on the algebra
$B(H)$ of all bounded linear operators on the infinite-dimensional
separable Hilbert space H. A similar description for the
finite-dimensional case appeared later in \cite{KK}. In the paper
\cite{LW} 2-local derivations have been described on matrix
algebras over finite-dimensional division rings. In \cite{AK} the
authors suggested a new technique and have generalized the above
mentioned results of \cite{S} and \cite{KK} for arbitrary Hilbert
spaces. Namely they considered 2-local derivations on the algebra
$B(H)$ of all linear bounded operators on an arbitrary (no
separability is assumed) Hilbert space $H$ and proved that every
2-local derivation on $B(H)$ is a derivation. In \cite{AA2},
\cite{AK2} the authors extended the above results and give a proof
of the theorem for arbitrary von Neumann algebras.

After a number of paper were devoted to 2-local
derivations, weak-2-local derivations, Weak-2-local
$*$-derivations, 2-local triple derivations, 2-local Lie isomorphisms,
2-local *-Lie isomorphisms and so on.

Results on 2-local derivations on finite dimensional Lie algebras were obtained
in \cite{AyuKudRak16}, \cite{LaiZheng15}. Articles \cite{CaPe2016}, \cite{NiPe2015}, \cite{NiPe2015b}
are devoted to weak-2-local derivations, and \cite{ChenHuangLu15}, \cite{LiLu16},
\cite{LiLu16b}, \cite{Liu16} are devoted to 2-local $*$-Lie isomorphisms and 2-local Lie isomorphisms.
A number of theorem on 2-local triple derivations were proved in \cite{HamKuPeRu}, \cite{KuOPeRu14}.
Other classes of 2-local maps on different types of associative and Jordan algebras were studied
in \cite{AyuKu16}, \cite{BurFerGarPe15}, \cite{BurFerGarPe15b}, \cite{CaPe2015}, \cite{ChenWang} and \cite{Pe15}.

In this article we develop an algebraic approach to
investigation of derivations and 2-local derivations on
associative and Jordan rings. Since we consider sufficiently general cases
of associative rings we restrict our attention only on inner
derivations and 2-local inner derivations. In particular, we
consider the following problem:  if a derivation on
an associative ring is a 2-local inner derivation then is this derivation
inner? The answer to this question is affirmative if the ring is
generated by two elements (Theorem \ref{10}).

In section 2 we consider 2-local derivations on the matrix
ring $M_n(\Re)$ over an associative ring $\Re$. It is proved that, given a
commutative ring $\Re$, an arbitrary 2-local inner
derivation on $M_n(\Re)$ is an inner derivation. This result
extends the one obtained in \cite{LW} to the infinite dimensional
case but for a commutative ring $\Re$ and in \cite{AK4} to the case of a
commutative ring but only for 2-local inner derivations.

In section 3 we show that every derivation on an
associative ring $\Re$ has an extension to a derivation on
the matrix ring $M_n(\Re)$ of $n\times n$ matrices over $\Re$.

In section 4 the relationship between 2-local derivations
and 2-local Jordan derivations on associative rings is studied.

In section 5 2-local derivations on the Jordan matrix ring over a commutative associative ring
are studied. Namely, we investigate 2-local inner derivations on the Jordan
ring $H_n(\Re)$ of $n$-dimensional matrices over a commutative associative ring $\Re$.
It is proved that every such 2-local inner derivation is
a derivation. For this propose we use a Jordan analog of the algebraic approach to the
investigation of 2-local derivations applied to matrix
rings over commutative associative rings developed in section 2. Thus
the method developed in this paper is sufficiently universal
and can also be applied to Jordan and Lie rings. Its respective
modification allows to prove similar problems for Jordan and Lie rings of
matrices over a $*$-ring.

\section{2-local derivations on matrix rings}

\medskip

Let $\Re$ be a ring. Recall that a map $D : \Re\to \Re$ is called
a derivation, if $D(x+y)=D(x)+D(y)$ and $D(xy)=D(x)y+xD(y)$ for
any two elements $x$, $y\in \Re$.

A map $\Delta : \Re\to \Re$ is called a 2-local derivation, if for
any two elements $x$, $y\in \Re$ there exists a derivation
$D_{x,y}:\Re\to \Re$ such that $\Delta (x)=D_{x,y}(x)$, $\Delta
(y)=D_{x,y}(y)$.

Now let $\Re$ be an associative ring. A derivation $D$ on $\Re$
is called an inner derivation, if there exists an element $a\in
\Re$ such that
$$
D(x)=ax-xa, x\in \Re.
$$
A map $\Delta : \Re\to \Re$ is called a 2-local inner derivation,
if for any two elements $x$, $y\in \Re$ there exists an element
$a\in \Re$ such that $\Delta (x)=ax-xa$, $\Delta (y)=ay-ya$.

Let $\Re$ be a unital associative ring, $M_n(\Re)$ be the
matrix ring over $\Re$, $n>1$, of matrices of the form
$$
\left[%
\begin{array}{cccc}
 a^{1,1} & a^{1,2} & \cdots & a^{1,n}\\
a^{2,1} & a^{2,2} & \cdots & a^{2,n}\\
\vdots & \vdots & \ddots & \vdots\\
a^{n,1} & a^{n,2} & \cdots & a^{n,n}\\
\end{array}%
\right],
a^{i,j}\in \Re, i,j=1,2,\dots ,n.
$$
Let $\{e_{i,j}\}_{i,j=1}^n$ be the set of matrix units in
$M_n(\Re)$, i.e. $e_{i,j}$ is a matrix with components
$a^{i,j}={\bf 1}$ and $a^{k,l}={\bf 0}$ if $(i,j)\neq(k,l)$, where
${\bf 1}$ is the identity element, ${\bf 0}$ is the zero element of
$\Re$ and a matrix $a\in M_n(\Re)$ is written as $a=\sum_{k,l=1}^n
a^{k,l}e_{k,l}$, where $a^{k,l}\in \Re$ for $k,l=1,2,\dots, n$.

\medskip

First, let us prove some lemmata which will be used in the
proof of Theorem \ref{1}. Throughout the paper, $\Re$ denotes a unital associative
ring where $2$ is invertible, $M_n(\Re)$ denotes the ring of $n\times
n$ matrices over $\Re$, $n>1$. Let $\Delta :M_n(\Re)\to M_n(\Re)$
be a 2-local inner derivation. Let us fix a subset
$\{a(i,j)\}_{i,j=1}^n\subset M_n(\Re)$  such that
$$
\Delta(e_{i,j})=a(i,j)e_{i,j}-e_{i,j}a(i,j),
$$
$$
\Delta(\sum_{k=1}^{n-1}e_{k,k+1})=a(i,j)(\sum_{k=1}^{n-1}e_{k,k+1})-(\sum_{k=1}^{n-1}e_{k,k+1})a(i,j).
$$
Put $a_{i,j}=e_{i,i}a(j,i)e_{j,j}$, for all pairs of distinct
indices $i$, $j$ and let $\sum_{k\neq l} a_{k,l}$ be the sum of
all such elements.

\begin{lemma} \label{21}
Let $\Delta :M_n(\Re)\to M_n(\Re)$ be a 2-local inner derivation.
Then the identity
$$
e_{k,k}a(i,j)e_{i,j}=e_{k,k}a(i,k)e_{i,j}
$$
holds for any pair $i$, $k$ of distinct indices from $\{1,2,\dots,n\}$ and for any
$j\in\{1,2,\dots,n\}$, and the equality
$$
e_{i,j}a(i,j)e_{k,k}=e_{i,j}a(k,j)e_{k,k}.
$$
holds for any pair $j$, $k$ of distinct indices from $\{1,2,\dots,n\}$ and for any
$i\in\{1,2,\dots,n\}$.
\end{lemma}

\begin{proof}
Let $d\in M_n(\Re)$ be such element that
$$
\Delta(e_{i,j})=de_{i,j}-e_{i,j}d,
\Delta(e_{i,k})=de_{i,k}-e_{i,k}d.
$$
Then
$$
a(i,j)e_{i,j}-e_{i,j}a(i,j)=de_{i,j}-e_{i,j}d,
$$
$$
a(i,k)e_{i,k}-e_{i,k}a(i,k)=de_{i,k}-e_{i,k}d,
$$
and
$$
e_{k,k}a(i,j)e_{i,j}=e_{k,k}de_{i,j},
e_{k,k}a(i,k)e_{i,k}=e_{k,k}de_{i,k}.
$$
Hence
$$
e_{k,k}a(i,k)e_{i,k}e_{k,j}=e_{k,k}de_{i,k}e_{k,j}=e_{k,k}de_{i,j}
$$
and
$$
e_{k,k}a(i,k)e_{i,j}=e_{k,k}de_{i,j}=e_{k,k}a(i,j)e_{i,j}.
$$
Similarly,
$$
e_{i,j}a(i,j)e_{k,k}=e_{i,j}a(k,j)e_{k,k}.
$$
\end{proof}

\medskip

\begin{lemma} \label{2}
Let $\Delta :M_n(\Re)\to M_n(\Re)$ be a 2-local inner derivation.
Then for any pair $i$, $j$ of distinct indices in
$\{1,2,\dots,n\}$ the following equality holds
$$
\Delta(e_{i,j})=(\sum_{k,l=1, k\neq
l}^na_{k,l})e_{i,j}-e_{i,j}(\sum_{k,l=1, k\neq
l}^na_{l,k})+a(i,j)^{i,i}e_{i,j}-e_{i,j}a(i,j)^{j,j},
$$
where $a(i,j)^{i,i}$, $a(i,j)^{j,j}$ are the appropriate
components of the matrices $e_{i,i}a(i,j)e_{i,i}$,
$e_{j,j}a(i,j)e_{j,j}$ respectively.
\end{lemma}

\begin{proof}
We have
$$
\Delta(e_{i,j})=a(i,j)e_{i,j}-e_{i,j}a(i,j)=\sum_{k=1}^ne_{k,k}
a(i,j)e_{i,j}-\sum_{k=1}^ne_{i,j}a(i,j)e_{k,k}
$$
$$
=\sum_{k=1, k\neq i}^ne_{k,k} a(i,j)e_{i,j}-\sum_{k=1, k\neq
j}^ne_{i,j}a(i,j)e_{k,k}+e_{i,i}a(i,j)e_{i,j}-e_{i,j}a(i,j)e_{j,j}
$$
$$
=\sum_{k=1, k\neq i}^ne_{k,k} a(i,k)e_{i,j}-\sum_{k=1, k\neq
j}^ne_{i,j}a(k,j)e_{k,k}+a(i,j)^{i,i}e_{i,j}-e_{i,j}a(i,j)^{j,j}
$$
$$
=\sum_{k=1, k\neq i}^na_{k,i}e_{i,j}-\sum_{k=1, k\neq
j}^ne_{i,j}a_{j,k}+a(i,j)^{i,i}e_{i,j}-e_{i,j}a(i,j)^{j,j}
$$
$$
=(\sum_{k,l=1, k\neq l}^na_{k,l})e_{i,j}-e_{i,j}(\sum_{k,l=1, k\neq
l}^na_{l,k})+a(i,j)^{i,i}e_{i,j}-e_{i,j}a(i,j)^{j,j}
$$
by Lemma \ref{21}.
\end{proof}

Let $x_o=\sum_{k=1}^{n-1}e_{k,k+1}$. From the hypothesis there exists an element
$c\in M_n(\Re)$ such that
$$
\Delta(x_o)=cx_o-x_oc.
$$
Let $c=\sum_{i,j=1}^nc_{i,j}$ be the decomposition of $c$ with respect to $\{e_{i,j}\}_{i,j=1}^n$,
where $c_{i,j}=e_{i,i}ce_{j,j}$, $i,j=1,2,\dots, n$.

\begin{lemma} \label{3}
Let $\Delta :M_n(\Re)\to M_n(\Re)$ be a 2-local inner derivation.
Let $k$, $l$ be an arbitrary couple of distinct numbers in $\{1,2,\dots,n\}$,
and let $b\in M_n(\Re)$ be an element such that
$$
\Delta(x_o)=bx_o-x_ob.
$$
Then $c^{k,k}-c^{l,l}=b^{k,k}-b^{l,l}$, where
$c_{i,i}=c^{i,i}e_{i,i}$, $b_{i,i}=b^{i,i}e_{i,i}$, $c^{i,i}$,
$b^{i,i}\in \Re$, $i=1,2,\dots,n$.
\end{lemma}

\begin{proof} We may assume that  $k<l$. We have
$$
\Delta(x_o)=cx_o-x_oc=bx_o-x_ob.
$$
Hence
$$
e_{k,k}(cx_o-x_oc)e_{k+1,k+1}=e_{k,k}(bx_o-x_ob)e_{k+1,k+1}
$$
and
$$
c^{k,k}-c^{k+1,k+1}=b^{k,k}-b^{k+1,k+1}.
$$
Then for the sequence
$$
(k,k+1),(k+1,k+2)\dots (l-1,l)
$$
we have
$$
c^{k,k}-c^{k+1,k+1}=b^{k,k}-b^{k+1,k+1},
c^{k+1,k+1}-c^{k+2,k+2}=b^{k+1,k+1}-b^{k+2,k+2},\dots
$$
$$
c^{l-1,l-1}-c^{l,l}=b^{l-1,l-1}-b^{l,l}.
$$
Hence
$$
c^{k,k}-b^{k,k}=c^{k+1,k+1}-b^{k+1,k+1},
c^{k+1,k+1}-b^{k+1,k+1}=c^{k+2,k+2}-b^{k+2,k+2},\dots
$$
$$
c^{l-1,l-1}-b^{l-1,l-1}=c^{l,l}-b^{l,l}.
$$
Therefore $c^{k,k}-b^{k,k}=c^{l,l}-b^{l,l}$, i.e.
$c^{k,k}-c^{l,l}=b^{k,k}-b^{l,l}$. The proof is complete.
\end{proof}

\medskip

Let $a_{i,i}=c_{i,i}$ for $i=,2,\dots,n$ and
$\bar{a}=\sum_{i,j=1}^na_{i,j}$.

The following theorem is the main result of the paper.

\begin{theorem} \label{1}
Let $\Re$ be a commutative associative unital ring, and let $M_n(\Re)$
be the ring of $n\times n$ matrices over $\Re$, $n>1$. Then
any 2-local inner derivation on the matrix ring $M_n(\Re)$ is an inner derivation.
\end{theorem}

\begin{proof}
Let $\Delta :M_n(\Re)\to M_n(\Re)$ be a 2-local inner derivation,
$x$ be an arbitrary matrix in $M_n(\Re)$. Let $\bar{a}$ be the element described in
previous paragraphs. We shall show that $\Delta(x)=\bar{a}x-x\bar{a}$. Let $d(i,j)\in
M_n(\Re)$ be an element such that
$$
\Delta(e_{i,j})=d(i,j)e_{i,j}-e_{i,j}d(i,j), \,\,\,\,
\Delta(x)=d(i,j)x-xd(i,j)
$$
and $i\neq j$. Then by Lemma \ref{2} we have
$$
\Delta(e_{i,j})=d(i,j)e_{i,j}-e_{i,j}d(i,j)
$$
$$
=e_{i,i}d(i,j)e_{i,j}-e_{i,j}d(i,j)e_{j,j}+
(1-e_{i,i})d(i,j)e_{i,j}-e_{i,j}d(i,j)(1-e_{j,j})
$$
$$
=a(i,j)_{i,i}e_{i,j}-e_{i,j}a(i,j)_{j,j}+(\sum_{k\neq l}
a_{k,l})e_{i,j}-e_{i,j}(\sum_{k\neq l} a_{k,l})
$$
for all $i$, $j$ in $\{1,2,\dots,n\}$.

Since
$e_{i,i}d(i,j)e_{i,j}-e_{i,j}d(i,j)e_{j,j}=a(i,j)_{i,i}e_{i,j}-e_{i,j}a(i,j)_{j,j}$
we have
$$
(1-e_{i,i})d(i,j)e_{i,i}=(\sum_{k\neq l} a_{k,l})e_{i,i},
$$
$$
e_{j,j}d(i,j)(1-e_{j,j})=e_{j,j}(\sum_{k\neq l} a_{k,l})
$$
for all pairs of distinct numbers $i$ and $j$ in $\{1,2,\dots,n\}$.

Hence
$$
e_{j,j}\Delta(x)e_{i,i}=e_{j,j}(d(i,j)x-xd(i,j))e_{i,i}
$$
$$
=e_{j,j}d(i,j)(1-e_{j,j})xe_{i,i}+
e_{j,j}d(i,j)e_{j,j}xe_{i,i}-e_{j,j}x(1-e_{i,i})d(i,j)e_{i,i}-e_{j,j}xe_{i,i}d(i,j)e_{i,i}
$$
$$
=e_{j,j}(\sum_{k\neq l} a_{k,l})xe_{i,i}-e_{j,j}x(\sum_{k\neq l}
a_{k,l})e_{i,i}+
e_{j,j}d(i,j)e_{j,j}xe_{i,i}-e_{j,j}xe_{i,i}d(i,j)e_{i,i}.
$$
We have
$$
\Delta(\sum_{k=1}^{n-1}e_{k,k+1})=a(i,j)(\sum_{k=1}^{n-1}e_{k,k+1})-(\sum_{k=1}^{n-1}e_{k,k+1})a(i,j)
$$
by the definition of $a(i,j)$. Then by Lemma \ref{3} we have
$$
a(i,j)^{j,j}-a(i,j)^{i,i}=c^{j,j}-c^{i,i},
$$
where
$$
c_{k,k}=c^{k,k}e_{k,k}, c^{k,k}\in \Re, k=1,2,\dots,n,
$$
$$
a(i,j)=\sum_{kl=1}^n a(i,j)^{k,l}e_{k,l}, a(i,j)^{k,l}\in \Re,
k,l=1,2,\dots,n.
$$
Since
$$
d(i,j)e_{i,j}-e_{i,j}d(i,j)=a(i,j)e_{i,j}-e_{i,j}a(i,j)
$$
we have
$$
e_{i,i}d(i,j)e_{i,j}-e_{i,j}d(i,j)e_{j,j}=e_{i,i}a(i,j)e_{i,j}-e_{i,j}a(i,j)e_{j,j}
$$
and
$$
(d(i,j)^{i,i}-d(i,j)^{j,j})e_{i,j}=(a(i,j)^{i,i}-a(i,j)^{j,j})e_{i,j},
$$
where $d(i,j)=\sum_{kl=1}^n d(i,j)^{k,l}e_{k,l}$.

Hence
$$
d(i,j)^{i,i}-d(i,j)^{j,j}=a(i,j)^{i,i}-a(i,j)^{j,j},
$$
i.e.
$$
d(i,j)^{j,j}-d(i,j)^{i,i}=a(i,j)^{j,j}-a(i,j)^{i,i}.
$$

Therefore
$$
e_{j,j}d(i,j)e_{j,j}xe_{i,i}-e_{j,j}xe_{i,i}d(i,j)e_{i,i}=d(i,j)^{j,j}x^{j,i}e_{j,i}-x^{j,i}d(i,j)^{i,i}e_{j,i}
$$
$$
=(d(i,j)^{j,j}-d(i,j)^{i,i})x^{j,i}e_{j,i}=(a(i,j)^{j,j}-a(i,j)^{i,i})x^{j,i}e_{j,i}
$$
$$
=(c^{j,j}-c^{i,i})x^{j,i}e_{j,i}=c^{j,j}x^{j,i}e_{j,i}-x^{j,i}c^{i,i}e_{j,i}=(c^{j,j}e_{j,j})e_{j,j}xe_{i,i}-e_{j,j}xe_{i,i}(c^{i,i}e_{i,i})
$$
$$
=a_{j,j}e_{j,j}xe_{i,i}-e_{j,j}xe_{i,i}a_{i,i},
$$
where $x=\sum_{kl=1}^n x^{k,l}e_{k,l}$.

Hence
$$
e_{j,j}\Delta(x)e_{i,i}=e_{j,j}(\sum_{k\neq l}
a_{k,l})xe_{i,i}-e_{j,j}x(\sum_{k\neq l} a_{k,l})e_{i,i}+
a_{j,j}e_{j,j}xe_{i,i}-e_{j,j}xe_{i,i}a_{i,i}
$$
$$
=e_{j,j}(\sum_{k\neq l} a_{k,l})xe_{i,i}-e_{j,j}x(\sum_{k\neq l}
a_{k,l})e_{i,i}+e_{j,j}(\sum_{k=1}^na_{k,k})xe_{i,i}-e_{j,j}x(\sum_{k=1}^na_{k,k})e_{i,i}
$$
$$
=e_{j,j}(\sum_{kl=1}^n a_{k,l})xe_{i,i}-e_{j,j}x(\sum_{kl=1}^n
a_{k,l})e_{i,i}=e_{j,j}(\bar{a}x-x\bar{a})e_{i,i}.
$$

Let $d(i,i)$, $v$, $w\in \mathcal{M}$ be elements such that
$$
\bigtriangleup(e_{i,i})=d(i,i)e_{i,i}-e_{i,i}d(i,i), \,\,\,\,
\bigtriangleup(x)=d(i,i)x-xd(i,i),
$$
$$
\bigtriangleup(e_{i,i})=ve_{i,i}-e_{i,i}v,
\bigtriangleup(e_{i,j})=ve_{i,j}-e_{i,j}v,
$$
and
$$
\bigtriangleup(e_{i,i})=we_{i,i}-e_{i,i}w,
\bigtriangleup(e_{j,i})=we_{j,i}-e_{j,i}w.
$$
Then
$$
(1-e_{i,i})a(i,j)e_{i,i}=(1-e_{i,i})ve_{i,i}=(1-e_{i,i})d(i,i)e_{i,i},
$$
and
$$
e_{i,i}a(j,i)(1-e_{i,i})=e_{i,i}w(1-e_{i,i})=e_{i,i}d(i,i)(1-e_{i,i}).
$$
By Lemma \ref{2} we have
$$
\Delta(e_{i,j})=a(i,j)e_{i,j}-e_{i,j}a(i,j)
$$
$$
=(\sum_{k\neq l} a_{k,l})e_{i,j}-e_{i,j}(\sum_{k\neq l}
a_{k,l})+a(i,j)_{i,i}e_{i,j}-e_{i,j}a(i,j)_{j,j}
$$
and
$$
(1-e_{i,i})a(i,j)e_{i,i}=(\sum_{k\neq l} a_{k,l})e_{i,i}.
$$
Similarly
$$
e_{i,i}a(j,i)(1-e_{i,i})=e_{i,i}(\sum_{k\neq l} a_{k,l}).
$$

Hence
$$
e_{i,i}\Delta(x)e_{i,i}=e_{i,i}(d(i,i)x-xd(i,i))e_{i,i}
$$
$$
=e_{i,i}d(i,i)(1-e_{i,i})xe_{i,i}+
e_{i,i}d(i,i)e_{i,i}xe_{i,i}-e_{i,i}x(1-e_{i,i})d(i,i)e_{i,i}-e_{i,i}xe_{i,i}d(i,i)e_{i,i}
$$
$$
=e_{i,i}a(j,i)(1-e_{i,i})xe_{i,i}+
e_{i,i}d(i,i)e_{i,i}xe_{i,i}-e_{i,i}x(1-e_{i,i})a(i,j)e_{i,i}-e_{i,i}xe_{i,i}d(i,i)e_{i,i}
$$
$$
=e_{i,i}(\sum_{k\neq l} a_{k,l})xe_{i,i}-e_{i,i}x(\sum_{k\neq l}
a_{k,l})e_{i,i}+
e_{i,i}d(i,i)e_{i,i}xe_{i,i}-e_{i,i}xe_{i,i}d(i,i)e_{i,i}
$$
$$
=e_{i,i}(\sum_{k\neq l} a_{k,l})xe_{i,i}-e_{i,i}x(\sum_{k\neq l}
a_{k,l})e_{i,i}+0
$$
$$
=e_{i,i}(\sum_{k\neq l} a_{k,l})xe_{i,i}-e_{i,i}x(\sum_{k\neq l}
a_{k,l})e_{i,i}+c_{i,i}e_{i,i}xe_{i,i}-e_{i,i}xc_{i,i}e_{i,i}
$$
$$
=e_{i,i}(\sum_{k\neq l} a_{k,l})xe_{i,i}-e_{i,i}x(\sum_{k\neq l}
a_{k,l})e_{i,i}+
$$
$$
e_{i,i}(\sum_{k=1}^na_{k,k})xe_{i,i}-e_{i,i}x(\sum_{k=1}^na_{k,k})e_{i,i}
$$
$$
=e_{i,i}(\sum_{kl=1}^n a_{k,l})xe_{i,i}-e_{i,i}x(\sum_{kl=1}^n
a_{k,l})e_{i,i}=e_{i,i}(\bar{a}x-x\bar{a})e_{i,i}.
$$

By the above conclusions we have
$$
\Delta(x)=\sum_{kl=1}^ne_{k,k}\Delta(x)e_{l,l}=\sum_{kl=1}^ne_{k,k}(\bar{a}x-x\bar{a})e_{l,l}=\bar{a}x-x\bar{a}
$$
for all $x\in M_n(\Re)$. The proof is complete.
\end{proof}

\bigskip

\section{On extensions of derivations and 2-local
derivations}

\begin{lemma} \label{5}
Let $M_2(\Re)$ be the ring of $2\times 2$ matrices over an
associative unital ring $\Re$ and let $D$ be a derivation
on the subring $\Re e_{1,1}$ and $\delta$ be a derivation on $\Re$
induced by $D$. Then the map
$$
\bar{D}
\left(\left[%
\begin{array}{lr}
  \lambda & \mu \\
  \nu & \eta \\
\end{array}%
\right]\right)=
\left[%
\begin{array}{lr}
  \delta(\lambda) & \delta(\mu)+\mu \\
  \delta(\nu)-\nu & \delta(\eta) \\
\end{array}%
\right], \lambda, \mu, \nu, \eta\in \Re,
$$
is a derivation.
\end{lemma}

\begin{proof}
It is easy to check that for $a,b\in M_2(\Re)$ we have
$\bar{D}(ab)=\bar{D}(a)b+a\bar{D}(b)$. Indeed, the map $\bar{D}$
is equal to $\bar{\delta}+d_U$, where
$$
\bar{\delta}\left(\left[%
\begin{array}{lr}
\lambda & \mu \\
\nu & \eta \\
\end{array}%
\right]\right)=
\left[%
\begin{array}{lr}
\delta(\lambda) & \delta(\mu) \\
  \delta(\nu) & \delta(\eta) \\
\end{array}%
\right], \lambda, \mu, \nu, \eta\in \Re,
$$
and $d_U$ is the inner derivation induced by the matrix
$$
\left[
\begin{array}{lr} \frac{1}{2}  & 0 \\ 0 & -\frac{1}{2}
\end{array} \right].
$$
\end{proof}

Let $\bar{M}_m(\Re)$ be the subring of $M_n(\Re)$, $m<n$, generated
by the subsets
$$
\Re e_{i,j}, i,j=1,2,\dots, m
$$
in $M_n(\Re)$. It is clear that
$$
\bar{M}_m(\Re)\cong M_m(\Re).
$$

\begin{lemma} \label{6}
Let $\Re$ be an associative ring, and let $M_n(\Re)$ be the
ring of $n\times n$ matrices over $\Re$, $n>2$. Then every
derivation on $\bar{M}_2(\Re)$ can be extended to a derivation on
$M_n(\Re)$.
\end{lemma}

\begin{proof} By Lemma \ref{5} every derivation on $\bar{M}_2(\Re)$
can be extended to a derivation on $M_4(\Re)$. In its turn, every
derivation on $\bar{M}_4(\Re)$ can be extended to a derivation on
$M_8(\Re)$ and so on. Thus every derivation $\partial$ on
$\bar{M}_2(\Re)$ can be extended to a derivation $D$ on
$M_{2^k}(\Re)$. Suppose that $n\leq 2^k$. Let $e=\sum_{i=1}^n
e_{i,i}$ and
$$
\bar{D}(a)=eD(a)e, a\in \bar{M}_n(\Re).
$$
Then $\bar{D}:\bar{M}_n(\Re)\to \bar{M}_n(\Re)$ and $\bar{D}$ is a
derivation on $\bar{M}_n(\Re)$ by \cite[Proposition 2.7]{NiPe2015}.
At the same time, the derivation
$\bar{D}$ coincides with the derivation $\partial$ on
$\bar{M}_2(\Re)$. Therefore, $\bar{D}$ is an extension of
$\partial$ to $\bar{M}_n(\Re)$. Hence every derivation $\partial$
on $\bar{M}_2(\Re)$ can be extended to a derivation on $M_n(\Re)$.
\end{proof}

Thus, in the case of the ring $M_2(\Re)$ for any derivation on the
subring $\Re e_{1,1}$ we can take its extension onto the whole
$M_2(\Re)$ defined as in Lemma \ref{5}, which is also a
derivation.

\begin{theorem} \label{7}
Let $\Re$ be an associative ring, and let $M_n(\Re)$ be the
ring of $n\times n$ matrices over $\Re$, $n>2$. Then every
derivation on $\Re$ can be extended to a derivation on
$M_n(\Re)$.
\end{theorem}

\begin{proof}
Let $\delta$ be an arbitrary derivation on $\Re$ and $D$ be the derivation on the subring $\Re e_{1,1}$
such that $\delta$ is induced by $D$. By Lemma \ref{5} every derivation on $\Re e_{1,1}$ has an extension
to a derivation on the matrix ring $\bar{M}_2(\Re)$ and every derivation on $\bar{M}_2(\Re)$
has an extension to a derivation on the matrix ring $M_n(\Re)$ by Lemma \ref{6}. Thus the statement of the
theorem is valid.
\end{proof}

\medskip

\begin{remark}
As for 2-local derivations, by \cite[Theorem 3.5]{AK4}
the lattice $P(\mathcal{M})$ of projections in a von Neumann algebra
$\mathcal{M}$ is not atomic if and only if the algebra $S(\mathcal{M})$ of
all measurable operators affiliated with $\mathcal{M}$ admits a 2-local derivation
which is not a derivation. Hence, if $\Re$ is the algebra $S(\mathcal{M})$ and
$P(\mathcal{M})$ is not atomic then by \cite[Theorem 4.3]{AK4} there are
2-local derivations on $\Re e_{1,1}$ which have no extension to a 2-local derivation on
$M_n(\Re)$, $n>1$.
\end{remark}

\medskip

We conclude the section by the following more general observation.

\begin{theorem} \label{10}
Let $\Delta :\Re\to \Re$ be a derivation on an
associative ring $\Re$. Suppose that $\Re$ is generated by its two
elements. Then, if $\Delta$ is a 2-local inner derivation then it is an inner
derivation.
\end{theorem}

\begin{proof}
Let $x$, $y$ be generators of $\Re$, i.e.
$\Re=Alg(\{x,y\})$, where $Alg(\{x,y\})$ is an associative ring,
generated by the elements $x$, $y$ in $\Re$. We have that there
exists $d\in \Re$ such that
$$
\Delta(x)=[d,x], \Delta(y)=[d,y],
$$
where $[d,a]=da-ad$ for any $a\in \Re$.

Hence by the additivity of $\Delta$ we have
$$
\Delta(x+y)=\Delta(x)+\Delta(y)=[d,x+y].
$$
Since $\Delta$ is a derivation we have
$$
\Delta(xy)=\Delta(x)y+x\Delta(y)=[d,x]y+x[d,y]=[d,xy],
$$
$$
\Delta(x^2)=\Delta(x)x+x\Delta(x)=[d,x]x+x[d,x]=[d,x^2],
$$
$$
\Delta(y^2)=\Delta(y)y+y\Delta(y)=[d,y]y+y[d,y]=[d,y^2],
$$
Similarly
$$
\Delta(x^k)=[d,x^k], \Delta(y^m)=[d,y^m],
\Delta(x^ky^m)=[d,x^ky^m]
$$
and
$$
\Delta(x^ky^mx^l)=\Delta(x^ky^m)x^l+x^ky^m\Delta(x^l)=[d,x^ky^m]x^l+x^ky^m[d,x^l]=[d,x^ky^mx^l].
$$
Finally, for every polynomial $p(x_1,x_2,\dots,x_m)\in \Re$, where
$x_1,x_2,\dots,x_m\in \{x,y\}$ we have
$$
\Delta(p(x_1,x_2,\dots,x_m))=[d,p(x_1,x_2,\dots,x_m)],
$$
i.e. $\Delta$ is an inner derivation on $\Re$.
\end{proof}

\section{2-local derivations on Jordan rings.}

This section is devoted to derivations and 2-local derivations of Jordan rings.

Given subsets $B$ and $C$ of a Lie algebra with bracket $[\cdot, \cdot]$, let $[B,C]$ denote
the set of all finite sums of elements $[b,c]$, where $b\in B$ and $c\in C$.

Consider a Jordan ring $\Re$ and let $m=\{xM: x\in \Re\}$, where $xM$  denotes the multiplication
operator defined by $(xM)y:=x\cdot y$ for all $x$, $y\in \Re$.
Let $aut(\Re)$ denotes the Lie ring of all derivations of $\Re$.
The elements of the ideal $int(\Re):=[m,m]$
of $aut(\Re)$ are called inner derivations of $\Re$.

Let $\Delta$ be a 2-local derivation of the Jordan ring $\Re$.
$\Delta$ is called a 2-local inner derivation, if for each pair of elements
$x$, $y\in \Re$ there is an inner derivation $D$ of $\Re$ such that $\Delta(x)=D(x)$, $\Delta(y)=D(y)$.
Let $\mathcal{A}$ be an associative unital ring. Suppose $2=1+1$ is invertible in $\mathcal{A}$.
Then the set $\mathcal{A}$ with respect to the operations of
addition and Jordan multiplication
$$
a\cdot b=\frac{1}{2}(ab+ba), a, b\in \mathcal{A}
$$
is a Jordan ring. This Jordan ring we will denote by $(\mathcal{A}, \cdot)$. For any
elements $a_1$, $a_2$, $\dots,$ $a_m$, $b_1$, $b_2$, $\dots,$ $b_m\in \mathcal{A}$ the map
$$
D(x)=\sum_{k=1}^m D_{a_k,b_k}(x)=\sum_{k=1}^m(a_k\cdot (b_k\cdot x)-b_k\cdot (a_k\cdot x)), x\in \mathcal{A}
$$
is a derivation. Therefore {\it every inner derivation of the Jordan ring $(\mathcal{A}, \cdot)$ is an inner derivation of the
associative ring $\mathcal{A}$.} And also it is easy to see, that every inner derivation of the form $D_{ab-ba}(x)=(ab-ba)x-x(ab-ba)$,
$x\in \mathcal{A}$ is an inner derivation of the Jordan ring $(\mathcal{A}, \cdot)$. Indeed, we have
$$
D_{ab-ba}(x)=D_{\frac{1}{4}4(ab-ba)}(x)
$$
$$
=\frac{1}{4}[(4ab-b4a)x-x(4ab-b4a)]=((4a)\cdot (b\cdot x)-b\cdot ((4a)\cdot x)).
$$
Let $\Delta$ be a 2-local inner derivation of the Jordan ring $(\mathcal{A}, \cdot)$. Then for every pair of elements $x$, $y\in \mathcal{A}$
there is an inner derivation $D$ of $(\mathcal{A}, \cdot)$ such that $\Delta(x)=D(x)$, $\Delta(y)=D(y)$.
But $D$ is also an inner derivation of the associative ring $\mathcal{A}$. Hence, $\Delta$ is a 2-local inner derivation
of the associative ring $\mathcal{A}$. So, {\it every 2-local inner derivation of the Jordan ring $(\mathcal{A}, \cdot)$
is a 2-local inner derivation of the associative ring $\mathcal{A}$.}

Now, let $\mathcal{A}$ be an involutive unital ring and $\mathcal{A}_ {sa}$ be the set of all self-adjoint elements of
the ring $\mathcal{A}$. Suppose $2$ is invertible in $\mathcal{A}$. Then, it is known that $(\mathcal{A} _{sa}, \cdot)$ is a Jordan ring. We take $a_1$, $a_2$, $\dots,$ $a_m$, $b_1$, $b_2$, $\dots,$ $b_m\in\mathcal{A}_{sa}$ and the inner derivation
$$
D(x)=\sum_{k=1}^m D_{a_k,b_k}(x)=\sum_{k=1}^m(a_k\cdot (b_k\cdot x)-b_k\cdot (a_k\cdot x)), x\in \mathcal{A}_{sa}.
$$
Then
$$
\sum_{k=1}^m D_{a_k,b_k}(x)=\sum_{k=1}^m D_{\frac{1}{4}[a_k,b_k]}(x), x\in \mathcal{A}_{sa}.
$$
At the same time the map
$$
\sum_{k=1}^m D_{\frac{1}{4}[a_k,b_k]}(x), x\in \mathcal{A}
$$
is an inner derivation on the $*$-ring $\mathcal{A}$ and it is an extension of the derivation $D$.
Therefore {\it every inner derivation of the Jordan ring $(\mathcal{A}_{sa}, \cdot)$ is extended to
an inner derivation of the $*$-ring $\mathcal{A}$}. Such extension of derivations on
a special Jordan algebra are considered in \cite{UH}. As to a 2-local inner derivation, in this case
it is possible discuss extension of a 2-local inner derivation of the Jordan ring $(\mathcal{A}_{sa}, \cdot)$
to a 2-local inner derivation of the involutive ring $\mathcal{A}$. However, till now it was not possible
to carry out such extension without additional conditions. This problem shows the importance of the main result
in the following section.

\section{2-local derivations on the Jordan ring of matrices over a commutative ring}

Throughout of this section let $\Re$ be a commutative unital ring, $M_n(\Re)$ be the associative ring of $n\times n$ matrices
over $\Re$. Suppose $2$ is invertible in $\Re$. In this case the set
$$
H_n(\Re)=\{
\left[%
\begin{array}{cccc}
 a^{1,1} & a^{1,2} & \cdots & a^{1,n}\\
a^{2,1} & a^{2,2} & \cdots & a^{2,n}\\
\vdots & \vdots & \ddots & \vdots\\
a^{n,1} & a^{n,2} & \cdots & a^{n,n}\\
\end{array}%
\right]
\in M_n(\Re):
a^{i,j}=a^{j,i},
i,j=1,2,\dots ,n\}
$$
is a Jordan ring with respect to the addition and the Jordan multiplication
$$
a\cdot b=\frac{1}{2}(ab+ba), a, b\in H_n(\Re).
$$
This Jordan ring is denoted by $H_n(\Re)$.
Let $\bar{e}_{i,j}=e_{i,j}+e_{j,i}$ and $\bar{a}_{i,j}=\{e_{i,i}ae_{j,j}\}=(e_{i,i}a)e_{j,j}+e_{i,i}(ae_{j,j})$
for every $a\in H_n(\Re)$ and distinct $i$, $j$ in $\{1,2,\dots ,n\}$.

\begin{lemma}  \label{3.0}
Let $D=\sum_{k=1}^mD_{a_k,b_k}$ be an inner derivation on $H_n(\Re)$,
generated by $a_1$, $a_2$, $\dots,$ $a_m$, $b_1$, $b_2$, $\dots,$ $b_m\in H_n(\Re)$.
Then
$$
\sum_{k=1}^m [a_k,b_k]^{i,i}=0, i=1,2,\dots,n.
$$
\end{lemma}

\begin{proof} Indeed, let $i$ be an arbitrary index in $\{1,2,\dots ,n\}$. Then
for every $k\in\{1,2,\dots ,m\}$ we have
$$
[a_k,b_k]^{i,i}=\sum_{l=1}^n a_k^{i,l}b_k^{l.i}-\sum_{l=1}^n b_k^{i,l}a_k^{l.i}=0
$$
since $a_k$ and $b_k$ are symmetric matrices. This completes the proof.
\end{proof}

\begin{lemma}  \label{3.1}
Let $\Delta$ be a 2-local derivation
on $H_2(\Re)$ and let $\sum_{k=1}^m D_{a_k,b_k}$, $\sum_{k=1}^m D_{c_k,d_k}$ be inner derivations on $H_2(\Re)$
such that
$$
\Delta (e_{1,1})=\sum_{k=1}^m D_{a_k,b_k}(e_{1,1})=\sum_{k=1}^mD_{c_k,d_k}(e_{1,1}).
$$
Then
$$
e_{1,1}(\sum_{k=1}^m [a_k,b_k])e_{2,2}=e_{1,1}(\sum_{k=1}^m [c_k,d_k])e_{2,2},
$$
$$
e_{2,2}((\sum_{k=1}^m [a_k,b_k]))e_{1,1}=e_{2,2}((\sum_{k=1}^m [c_k,d_k]))e_{1,1}.
$$
\end{lemma}

\begin{proof}
We have
$$
\sum_{k=1}^m D_{a_k,b_k}(e_{1,1})=D_{\frac{1}{4}\sum_{k=1}^m [a_k,b_k]}(e_{1,1}),
$$
$$
\sum_{k=1}^m D_{c_k,d_k}(e_{1,1})=D_{\frac{1}{4}\sum_{k=1}^m [c_k,d_k]}(e_{1,1}).
$$
Hence
$$
D_{\frac{1}{4}\sum_{k=1}^m [a_k,b_k]}(e_{1,1})=D_{\frac{1}{4}\sum_{k=1}^m [c_k,d_k]}(e_{1,1}).
$$
Therefore from
$$
D_{\frac{1}{4}\sum_{k=1}^m [a_k,b_k]}(e_{1,1})=\frac{1}{4}(\sum_{k=1}^m [a_k,b_k])e_{1,1}-e_{1,1}\frac{1}{4}(\sum_{k=1}^m [a_k,b_k])
$$
$$
=\frac{1}{4}(\sum_{k=1}^m [c_k,d_k])e_{1,1}-e_{1,1}\frac{1}{4}(\sum_{k=1}^m [c_k,d_k])=D_{\frac{1}{4}\sum_{k=1}^m [c_k,d_k]}(e_{1,1})
$$
it follows that
$$
(\sum_{k=1}^m [a_k,b_k])e_{1,1}-e_{1,1}(\sum_{k=1}^m [a_k,b_k])=(\sum_{k=1}^m [c_k,d_k])e_{1,1}-e_{1,1}(\sum_{k=1}^m [c_k,d_k])
$$
and
$$
e_{2,2}(\sum_{k=1}^m [a_k,b_k])e_{1,1}=e_{2,2}(\sum_{k=1}^m [c_k,d_k])e_{1,1},
$$
$$
e_{1,1}(\sum_{k=1}^m [a_k,b_k])e_{2,2}=e_{1,1}(\sum_{k=1}^m [c_k,d_k])e_{2,2}.
$$
This completes the proof.
\end{proof}

\begin{theorem}  \label{3.3}
Every 2-local inner derivation on $H_2(\Re)$ is an inner derivation.
\end{theorem}

\begin{proof}
Let $\Delta$ be an arbitrary 2-local inner derivation on $H_2(\Re)$ and
$\sum_{k=1}^m D_{a_k,b_k}$ be an inner derivation on $H_2(\Re)$ such that
$$
\Delta (e_{1,1})=\sum_{k=1}^m D_{a_k,b_k}(e_{1,1}).
$$
We prove that
$$
\Delta (x)=D_{\frac{1}{4}\sum_{k=1}^m [a_k,b_k]}(x), x\in H_2(\Re).
$$

Let $x\in H_2(\Re)$ and $c_1$, $c_2$, $\dots,$ $c_m$, $d_1$, $d_2$, $\dots,$ $d_m$ be elements in $H_{2}(\Re)$ such that
$$
\Delta (e_{1,1})=\sum_{k=1}^m D_{c_k,d_k}(e_{1,1}), \Delta (x)=\sum_{k=1}^m D_{c_k,d_k}(x).
$$
Let $a=\frac{1}{4}\sum_{k=1}^m [a_k,b_k]$ and $d=\frac{1}{4}\sum_{k=1}^m [c_k,d_k]$.
Then by Lemma \ref{3.1} we have
$$
\Delta (x)=\sum_{k=1}^m D_{c_k,d_k}(x)=D_{\frac{1}{4}\sum_{k=1}^m [c_k,d_k]}(x)=dx-xd
$$
$$
=e_{1,1}de_{1,1}x+e_{1,1}de_{2,2}x+e_{2,2}de_{1,1}x+e_{2,2}de_{2,2}x
$$
$$
-xe_{1,1}de_{1,1}-xe_{1,1}de_{2,2}-xe_{2,2}de_{1,1}-xe_{2,2}de_{2,2}
$$
$$
=e_{1,1}de_{1,1}x+e_{1,1}ae_{2,2}x+e_{2,2}ae_{1,1}x+e_{2,2}de_{2,2}x
$$
$$
-xe_{1,1}de_{1,1}-xe_{1,1}ae_{2,2}-xe_{2,2}ae_{1,1}-xe_{2,2}de_{2,2}.
$$
Here we have
$$
e_{1,1}de_{1,1}x+e_{2,2}de_{2,2}x-xe_{1,1}de_{1,1}-xe_{2,2}de_{2,2}
$$
$$
=\sum_{k=1}^2(e_{1,1}de_{1,1}e_{1,1}xe_{k,k}-e_{k,k}xe_{1,1}e_{1,1}de_{1,1})
$$
$$
+\sum_{k=1}^2(e_{2,2}de_{2,2}e_{2,2}xe_{k,k}-e_{k,k}xe_{2,2}e_{2,2}de_{2,2})
$$
$$
=\sum_{k=1}^2(d^{1,1}x^{1,k}e_{1,k}-x^{k,1}d^{1,1}e_{k,1})+
\sum_{k=1}^2(d^{2,2}x^{2,k}e_{2,k}-x^{k,2}d^{2,2}e_{k,2})
$$
$$
=d^{1,1}x^{1,2}e_{1,2}-x^{2,1}d^{1,1}e_{2,1}+
d^{2,2}x^{2,1}e_{2,1}-x^{1,2}d^{2,2}e_{1,2}
$$
$$
=(d^{1,1}-d^{2,2})x^{1,2}e_{1,2}+
(d^{2,2}-d^{1,1})x^{2,1}e_{2,1}=0
$$
$$
=(a^{1,1}-a^{2,2})x^{1,2}e_{1,2}+
(a^{2,2}-a^{1,1})x^{2,1}e_{2,1}=\dots
$$
$$
=e_{1,1}ae_{1,1}x+e_{2,2}ae_{2,2}x-xe_{1,1}ae_{1,1}-xe_{2,2}ae_{2,2}
$$
by Lemma \ref{3.0}. Hence
$$
\Delta (x)=e_{1,1}de_{1,1}x+e_{1,1}ae_{2,2}x+e_{2,2}ae_{1,1}x+e_{2,2}de_{2,2}x
$$
$$
-xe_{1,1}de_{1,1}-xe_{1,1}ae_{2,2}-xe_{2,2}ae_{1,1}-xe_{2,2}de_{2,2}
$$
$$
=e_{1,1}ae_{1,1}x+e_{1,1}ae_{2,2}x+e_{2,2}ae_{1,1}x+e_{2,2}ae_{2,2}x
$$
$$
-xe_{1,1}ae_{1,1}-xe_{1,1}ae_{2,2}-xe_{2,2}ae_{1,1}-xe_{2,2}ae_{2,2}=D(x).
$$

Hence
$$
\Delta (x)=\sum_{k=1}^m D_{a_k,b_k}(x)=\sum_{k=1}^m (a_k(b_kx)-b_k(a_kx)), x\in H_2(\Re). \,\,\,\,\,\,\,\,\,\,(1)
$$
From (1) it follows that $\Delta$ is linear and
$$
\Delta (xy)=\Delta (x)y+x\Delta (y)
$$
for all elements $x$, $y\in H_2(\Re)$ with respect to the Jordan multiplication.
Hence $\Delta$ is an inner derivation. The proof is complete.
\end{proof}

Now, we prove Theorem \ref{3.3} for $H_n(\Re)$ with an arbitrary natural number $n>1$.
Throughout the rest part of the paper let $\Delta$ be an arbitrary but fixed 2-local inner derivation
on $H_n(\Re)$. Then we have the following lemma.

\begin{lemma}  \label{3.4}
Let $i$, $j$ be arbitrary distinct indices, $e=e_{i,i}+e_{j,j}$ and
$\Delta(e_{i,i})=\sum_{k=1}^m D_{a_k,b_k}(e_{i,i})$ for some $a_1$, $a_2$, $\dots,$ $a_m$, $b_1$, $b_2$, $\dots,$ $b_m$ in $H_n(\Re)$. Then
the mapping
$$
\Delta_{i,j}(x)=e\Delta(x)e, x\in eH_n(\Re)e
$$
is a derivation on $eH_n(\Re)e$ and
$$
\Delta_{i,j}(x)=\frac{1}{4}e(\sum_{k=1}^m [a_k,b_k])ex-\frac{1}{4}xe(\sum_{k=1}^m [a_k,b_k])e, x\in eH_n(\Re)e.
$$
\end{lemma}

\begin{proof}
Similar to proof of \cite[Proposition 2.7]{NiPe2015} it can be proved that
$\Delta_{i,j}$ is a 2-local derivation.
Let $x$ be an arbitrary element in $eH_n(\Re)e$ and
$$
\Delta(e_{i,i})=\sum_{k=1}^m D_{c_k,d_k}(e_{i,i}), \Delta(x)=\sum_{k=1}^m D_{c_k,d_k}(x).
$$
Similar to Lemma \ref{3.1} we have
$$
e_{i,i}[a,b]e_{j,j}=e_{i,i}[c,d]e_{j,j}, e_{j,j}[a,b]e_{i,i}=e_{j,j}[c,d]e_{i,i}.
$$
The rest part of the proof repeats the proof of Theorem \ref{3.3} for
$e_{i,i}$ and $e_{j,j}$ instead of $e_{1,1}$ and $e_{2,2}$ respectively.
The proof is complete.
\end{proof}

Let $a_1$, $a_2$, $\dots,$ $a_m$, $b_1$, $b_2$, $\dots,$ $b_m$, $d(ii)$ be elements in $H_n(\Re)$ such that
$$
\Delta(e_{i,i})=\sum_{k=1}^m D_{a_k,b_k}(e_{i,i}), d(ii)=\frac{1}{4}\sum_{k=1}^m [a_k,b_k].
$$
Under this notations we have the following lemma.

\begin{lemma} \label{3.41}
For each pair $i$, $j$ of indices the following equalities are valid
$$
e_{i,i}d(ii)e_{i,i}=e_{j,j}d(ii)e_{j,j}, e_{i,i}d(ii)e_{i,i}=e_{j,j}d(jj)e_{j,j},
$$
$$
e_{i,i}d(ii)e_{j,j}=e_{i,i}d(jj)e_{j,j}, e_{j,j}d(ii)e_{i,i}=e_{j,j}d(jj)e_{i,i},
$$
and for every $k\neq i,j$
$$
e_{i,i}d(ii)e_{k,k}=e_{i,i}d(jj)e_{k,k}, e_{k,k}d(ii)e_{j,j}=e_{k,k}d(jj)e_{j,j},
$$
$$
e_{j,j}d(ii)e_{k,k}=e_{j,j}d(jj)e_{k,k}, e_{k,k}d(ii)e_{j,j}=e_{k,k}d(jj)e_{j,j}.
$$
\end{lemma}

\begin{proof}
The equality
$$
e_{i,i}d(ii)e_{i,i}=e_{j,j}d(ii)e_{j,j},\,\,\,\,\,\,(2)
$$
is proved similar to Lemma \ref{3.0}.
The equality
$$
e_{i,i}d(ii)e_{i,i}=e_{j,j}d(jj)e_{j,j}
$$
follows from Lemma \ref{3.4} and equality (2). We have
$$
d(ii)e_{i,i}-e_{i,i}d(ii)=d(jj)e_{i,i}-e_{i,i}d(jj)
$$
by Lemma \ref{3.4}. Hence
$$
(d(ii)e_{i,i}-e_{i,i}d(ii))e_{j,j}=(d(jj)e_{i,i}-e_{i,i}d(jj))e_{j,j},
$$
$$
e_{j,j}(d(ii)e_{i,i}-e_{i,i}d(ii))=e_{j,j}(d(jj)e_{i,i}-e_{i,i}d(jj))
$$
and
$$
e_{i,i}d(ii)e_{j,j}=e_{i,i}d(jj)e_{j,j}, e_{j,j}d(ii)e_{i,i}=e_{j,j}d(jj)e_{i,i}.
$$
Similarly, from
$$
d(ii)e_{j,j}-e_{j,j}d(ii)=d(jj)e_{j,j}-e_{j,j}d(jj)
$$
it follows that
$$
e_{j,j}d(ii)e_{k,k}=e_{j,j}d(jj)e_{k,k}, e_{k,k}d(ii)e_{j,j}=e_{k,k}d(jj)e_{j,j}.
$$
The proof is complete.
\end{proof}

Let $i$, $j$ be arbitrary distinct indices from
$\{1,2,\dots,n\}$, let
$$
a_{i,i}=e_{i,i}d(ii)e_{i,i}, a^{i,i}e_{i,i}=d(ii)^{i,i}e_{i,i}, a^{i,i}\in \Re, d(ii)^{i,i}\in \Re,
$$
$$
a_{i,j}=e_{i,i}d(ii)e_{j,j}, a^{i,j}e_{i,j}=d(ii)^{i,j}e_{i,j}, a^{i,j}\in \Re, d(ii)^{i,j}\in \Re
$$
and
$$
\bar{a}=\sum_{k,l=1}^n a_{k,l}.
$$
By Lemma \ref{3.41} $\bar{a}$ is defined correctly.

\begin{lemma} \label{3.5}
For each pair $i$, $j$ of distinct indices
the following equality is valid
$$
\bigtriangleup(\bar{e}_{i,j})=\bar{a}\bar{e}_{i,j}-\bar{e}_{i,j}\bar{a}.    \,\,\,\,\,\,\,(3)
$$
\end{lemma}

\begin{proof}
Let $k$ be an arbitrary index different from $i$
and let $c_1$, $c_2$, $\dots,$ $c_m$, $d_1$, $d_2$, $\dots,$ $d_m\in H_n(\Re)$ be elements such that
$$
\Delta (\bar{e}_{i,j})=\sum_{k=1}^m D_{c_k,d_k}(\bar{e}_{i,j}), \Delta (e_{k,k})=\sum_{k=1}^m D_{c_k,d_k}(e_{k,k}).
$$
Then
$$
\Delta (\bar{e}_{i,j})=D_{\frac{1}{4}\sum_{k=1}^m [c_k,d_k]}(\bar{e}_{i,j}), \Delta (e_{k,k})=D_{\frac{1}{4}\sum_{k=1}^m [c_k,d_k]}(e_{k,k})
$$
and
$$
e_{k,k}\Delta(\bar{e}_{i,j})e_{i,i}=e_{k,k}D_{\frac{1}{4}\sum_{k=1}^m [c_k,d_k]}(\bar{e}_{i,j})e_{i,i}=
e_{k,k}\frac{1}{4}(\sum_{k=1}^m [c_k,d_k])e_{j,i}-0
$$
$$
=e_{k,k}d(kk)e_{j,i}-0=e_{k,k}a_{k,j}e_{j,i}-e_{k,k}\bar{e}_{i,j}(\sum_{k,l=1}^n a_{k,l})e_{i,i}
$$
$$
=e_{k,k}(\sum_{k,l=1}^n a_{k,l})\bar{e}_{i,j}e_{i,i}-e_{k,k}\bar{e}_{i,j}(\sum_{k,l=1}^n a_{k,l})e_{i,i}
$$
$$
=e_{k,k}[\bar{a}\bar{e}_{i,j}-\bar{e}_{i,j}\bar{a}]e_{i,i}.
$$
Similarly,
$$
e_{i,i}\Delta(\bar{e}_{i,j})e_{k,k}=
e_{i,i}(\bar{a}\bar{e}_{i,j}-\bar{e}_{i,j}\bar{a})e_{k,k}, e_{i,i}\Delta(\bar{e}_{i,j})e_{i,i}=
e_{i,i}(\bar{a}e_{i,i}-e_{i,i}\bar{a})e_{i,i}.
$$
Since $i$ and $j$ are mutually symmetric we have
$$
e_{j,j}\Delta(\bar{e}_{i,j})e_{k,k}=
e_{j,j}(\bar{a}\bar{e}_{i,j}-\bar{e}_{i,j}\bar{a})e_{k,k}, e_{k,k}\Delta(\bar{e}_{i,j})e_{j,j}=
e_{k,k}(\bar{a}\bar{e}_{i,j}-\bar{e}_{i,j}\bar{a})e_{j,j}
$$
and
$$
e_{j,j}\Delta(\bar{e}_{i,j})e_{j,j}=
e_{j,j}(\bar{a}e_{i,i}-e_{i,i}\bar{a})e_{j,j}.
$$
Hence
$$
\Delta (\bar{e}_{i,j})=\sum_{k,l=1}^n e_{k,k}(\Delta (\bar{e}_{i,j}))e_{l,l}
$$
$$
=\sum_{k=1}^n e_{k,k}(\Delta (\bar{e}_{i,j}))e_{i,i}+\sum_{l=1}^n e_{i,i}(\Delta (\bar{e}_{i,j}))e_{l,l}
$$
$$
+\sum_{k=1}^n e_{k,k}(\Delta (\bar{e}_{i,j}))e_{j,j}+\sum_{l=1}^n e_{j,j}(\Delta (\bar{e}_{i,j}))e_{l,l}+
\sum_{k,l=1, k,l\neq i,j}^n e_{k,k}(\Delta (\bar{e}_{i,j}))e_{l,l}
$$
$$
=\sum_{k=1}^n e_{k,k}[\bar{a}\bar{e}_{i,j}-\bar{e}_{i,j}\bar{a}]e_{i,i}+\sum_{l=1}^n e_{i,i}[\bar{a}\bar{e}_{i,j}-\bar{e}_{i,j}\bar{a}]e_{l,l}
$$
$$
+\sum_{k=1}^n e_{k,k}[\bar{a}\bar{e}_{i,j}-\bar{e}_{i,j}\bar{a}]e_{j,j}+\sum_{l=1}^n e_{j,j}[\bar{a}\bar{e}_{i,j}-\bar{e}_{i,j}\bar{a}]e_{l,l}+
\sum_{k,l=1, k,l\neq i,j}^n e_{k,k}[\bar{a}\bar{e}_{i,j}-\bar{e}_{i,j}\bar{a}]e_{l,l}
$$
$$
=\bar{a}\bar{e}_{i,j}-\bar{e}_{i,j}\bar{a}.
$$
Hence the equality (3) is valid. This completes the proof.
\end{proof}

\medskip

\begin{theorem} \label{3.11}
Every 2-local inner derivation on $H_n(\Re)$ is a derivation.
\end{theorem}

\begin{proof}
We prove that the 2-local inner derivation $\Delta$ on $H_n(\Re)$ satisfies
the condition
$$
\Delta (x)=D_{\bar{a}}(x)=\bar{a}x-x\bar{a}, x\in H_n(\Re)).
$$

Let $x$ be an arbitrary element in $H_n(\Re)$ and let
$c_1$, $c_2$, $\dots,$ $c_m$, $d_1$, $d_2$, $\dots,$ $d_m\in H_n(\Re)$ be elements such that
$$
\Delta (e_{i,i})=\sum_{k=1}^m D_{c_k,d_k}(e_{i,i}), \Delta (x)=\sum_{k=1}^m D_{c_k,d_k}(x).
$$
Then
$$
\Delta (e_{i,i})=D_{\frac{1}{4}\sum_{k=1}^m [c_k,d_k]}(e_{i,i}),
\Delta (\bar{e}_{i,j})=D_{\frac{1}{4}\sum_{k=1}^m [c_k,d_k]}(\bar{e}_{i,j}),
$$
$$
\Delta (x)=D_{\frac{1}{4}\sum_{k=1}^m [c_k,d_k]}(x)
$$
by Lemma \ref{3.4}. Let $d=\frac{1}{4}\sum_{k=1}^m [c_k,d_k]$. Then
$$
\Delta(\bar{e}_{i,j})=de_{i,i}-e_{i,i}d \,\, \text{and}\,\,
\Delta(x)=dx-xd.
$$
By Lemma \ref{3.5} we have the following equalities
$$
\Delta(\bar{e}_{i,j})=d\bar{e}_{i,j}-\bar{e}_{i,j}d=
(e_{i,i}+e_{j,j})d\bar{e}_{i,j}-\bar{e}_{i,j}d(e_{i,i}+e_{j,j})
$$
$$
+(1-(e_{i,i}+e_{j,j}))d\bar{e}_{i,j}-\bar{e}_{i,j}d(1-(e_{i,i}+e_{j,j}))=
\bar{a}\bar{e}_{i,j}-\bar{e}_{i,j}\bar{a}
$$
for all $i$. We have
$$
(e_{i,i}+e_{j,j})d\bar{e}_{i,j}-\bar{e}_{i,j}d(e_{i,i}+e_{j,j})=
(e_{i,i}+e_{j,j})\bar{a}\bar{e}_{i,j}-\bar{e}_{i,j}\bar{a}(e_{i,i}+e_{j,j})
$$
and
$$
(1-(e_{i,i}+e_{j,j}))de_{i,i}=(1-(e_{i,i}+e_{j,j}))\bar{a}e_{i,i},
$$
$$
e_{i,i}d(1-(e_{i,i}+e_{j,j}))=e_{i,i}\bar{a}(1-(e_{i,i}+e_{j,j}))
$$
for all $i$. Also we have
$$
a_{i,j}=e_{i,i}de_{j,j}, a_{j,i}=e_{j,j}de_{i,i}
$$
by Lemma \ref{3.41}, and
$$
e_{i,i}de_{j,j}=e_{i,i}\bar{a}e_{j,j}, e_{j,j}de_{i,i}=e_{j,j}\bar{a}e_{i,i}.
$$
Hence
$$
(1-e_{i,i})de_{i,i}=(1-e_{i,i})\bar{a}e_{i,i},
e_{i,i}d(1-e_{i,i})=e_{i,i}\bar{a}(1-e_{i,i})\,\,\,\,\,\,\,\,\,\,(4)
$$
for all $i$. Therefore we have
$$
\{e_{i,i}\Delta(x)e_{j,j}\}=\frac{1}{2}[e_{i,i}\Delta(x)e_{j,j}+e_{j,j}\Delta(x)e_{i,i}]
$$
$$
=\frac{1}{2}[e_{i,i}(dx-xd)e_{j,j}+e_{j,j}(dx-xd)e_{i,i}]
$$
$$
=\frac{1}{2}[e_{i,i}d(1-e_{i,i})xe_{j,j}+
e_{i,i}de_{i,i}xe_{j,j}-e_{i,i}x(1-e_{j,j})de_{j,j}-e_{i,i}xe_{j,j}de_{j,j}
$$
$$
+e_{j,j}d(1-e_{j,j})xe_{i,i}+
e_{j,j}de_{j,j}xe_{i,i}-e_{j,j}x(1-e_{i,i})de_{i,i}-e_{j,j}xe_{i,i}de_{i,i}]
$$
$$
=\frac{1}{2}[e_{i,i}\bar{a}(1-e_{i,i})xe_{j,j}+
e_{i,i}de_{i,i}xe_{j,j}-e_{i,i}x(1-e_{j,j})\bar{a}e_{j,j}-e_{i,i}xe_{j,j}de_{j,j}
$$
$$
+e_{j,j}\bar{a}(1-e_{j,j})xe_{i,i}+
e_{j,j}de_{j,j}xe_{i,i}-e_{j,j}x(1-e_{i,i})\bar{a}e_{i,i}-e_{j,j}xe_{i,i}de_{i,i}]
$$
$$
=\frac{1}{2}[e_{i,i}\bar{a}(1-e_{i,i})xe_{j,j}-e_{i,i}x(1-e_{j,j})\bar{a}e_{j,j}+
e_{i,i}d(ii)e_{i,i}xe_{j,j}-e_{i,i}xe_{j,j}d(ii)e_{j,j}
$$
$$
+e_{j,j}\bar{a}(1-e_{j,j})xe_{i,i}-e_{j,j}x(1-e_{i,i})\bar{a}e_{i,i}+
e_{j,j}d(ii)e_{j,j}xe_{i,i}-e_{j,j}xe_{i,i}d(ii)e_{i,i}]
$$
$$
=\frac{1}{2}[e_{i,i}\bar{a}(1-e_{i,i})xe_{j,j}-e_{i,i}x(1-e_{j,j})\bar{a}e_{j,j}+
a_{i,i}xe_{j,j}-e_{i,i}xa_{j,j}
$$
$$
+e_{j,j}\bar{a}(1-e_{j,j})xe_{i,i}-e_{j,j}x(1-e_{i,i})\bar{a}e_{i,i}+
a_{j,j}xe_{i,i}-e_{j,j}xa_{i,i}]
$$
$$
=\{e_{jj}(\bar{a}x-x\bar{a})e_{ii}\}
$$
by Lemma \ref{3.41}. Also by equalities (4) we have
$$
e_{ii}\bigtriangleup(x)e_{ii}=e_{ii}(dx-xd)e_{ii}
$$
$$
=e_{ii}d(1-e_{ii})xe_{ii}+
e_{ii}de_{ii}xe_{ii}-e_{ii}x(1-e_{ii})de_{ii}-e_{ii}xe_{ii}de_{ii}
$$
$$
=e_{ii}\bar{a}(1-e_{ii})xe_{ii}+
e_{ii}de_{ii}xe_{ii}-e_{ii}x(1-e_{ii})\bar{a}e_{ii}-e_{ii}xe_{ii}de_{ii}
$$
$$
=e_{ii}\bar{a}(1-e_{ii})xe_{ii}+
e_{ii}d(ii)e_{ii}xe_{ii}-e_{ii}x(1-e_{ii})\bar{a}e_{ii}-e_{ii}xe_{ii}d(ii)e_{ii}
$$
$$
=e_{ii}(\bar{a}x-x\bar{a})e_{ii}.
$$
Hence
$$
\Delta(x)=\bar{a}x-x\bar{a}
$$
for all $x\in H_n(\Re)$, and $\Delta$ is a derivation on $H_n(\Re)$. Therefore
$\Delta$ is a derivation on $H_n(\Re)$. Since $\Delta$ is
chosen arbitrarily every 2-local inner derivation on $H_n{\Re}$
is a derivation. The proof is complete.
\end{proof}

\end{document}